\documentclass[11pt]{amsart}%
\usepackage{amssymb}
\usepackage{amsmath}
\usepackage{amsfonts}
\usepackage{graphicx}%
\providecommand{\U}[1]{\protect\rule{.1in}{.1in}}
\newtheorem{theorem}{Theorem}

\newtheorem{corollary}[theorem]{Corollary}

\newtheorem{definition}[theorem]{Definition}

\newtheorem{lemma}[theorem]{Lemma}

\newtheorem{proposition}[theorem]{Proposition}

\begin{document}
\title[A Moser/Bernstein type theorem]{A Moser/Bernstein type theorem in a Lie group with a left invariant metric
under a gradient decay condition}
\author[Aiolfi, Bonorino, Ripoll, Soret, Ville]{Ari Aiolfi, Leonardo Bonorino, Jaime Ripoll, \\ Marc Soret, Marina Ville}
\maketitle

\begin{abstract}
We say that a PDE in a Riemannian manifold $M$ is \emph{geometric }
if,$\ $whenever $u$ is a solution of the PDE on a domain $\Omega$ of $M$, the
composition $u_{\phi}:=u\circ\phi$ is also solution on $\phi^{-1}\left(
\Omega\right)  $, for any isometry $\phi$ of $M.$ We prove that if $u\in
C^{1}\left(  \mathbb{H}^{n}\right)  $ is a solution of a geometric PDE
satisfying the comparison principle, where $\mathbb{H}^{n}$ is the hyperbolic
space of constant sectional curvature $-1,$ $n\geq2,$ and if
\[
\limsup_{R\rightarrow\infty}\left(  e^{R}\sup_{S_{R}}\left\Vert \nabla
u\right\Vert \right)  =0,
\]
where $S_{R}$ is a geodesic sphere of $\mathbb{H}^{n}$ centered at fixed point
$o\in\mathbb{H}^{n}$ with radius $R,$ then $u$ is constant. Moreover, given
$C>0,$ there is a bounded non-constant harmonic function $v\in C^{\infty
}\left(  \mathbb{H}^{n}\right)  $ such that
\[
\lim_{R\rightarrow\infty}\left(  e^{R}\sup_{S_{R}}\left\Vert \nabla
v\right\Vert \right)  =C.
\]
The first part of the above result is a consequence of a more general theorem
proved in the paper which asserts that if $G$ is a non compact Lie group with
a left invariant metric, $u\in C^{1}\left(  G\right)  $ a solution of a left
invariant PDE (that is, if $v$ is a solution of the PDE on a domain $\Omega$
of $G$, the composition $v_{g}:=v\circ L_{g}$ of $v$ with a left translation
$L_{g}:G\rightarrow G,$ $L_{g}\left(  h\right)  =gh,$ is also solution on
$L_{g}^{-1}\left(  \Omega\right)  $ for any $g\in G),$ the PDE satisfies the
comparison principle and%
\[
\limsup_{R\rightarrow\infty}\left(  \sup_{g\in B_{R}}\left\Vert
\operatorname*{Ad}\nolimits_{g}\right\Vert \sup_{S_{R}}\left\Vert \nabla
u\right\Vert \right)  =0,
\]
where $\operatorname*{Ad}\nolimits_{g}:\mathfrak{g}\rightarrow\mathfrak{g}$ is
the adjoint map of $G$ and $\mathfrak{g}$ the Lie algebra of $G,$ then $u$ is constant.

\end{abstract}

\section{Introduction}

\qquad Finding a characterization of entire minimal graphs on $\mathbb{R}^{n}$
is a long standing problem on Differential Geometry. The first well known
result, due to Bernstein \cite{Be}, asserts that the only entire solutions of
the minimal surface equation (MSE) in $\mathbb{R}^{2}$ are affine functions.
J. Simons \cite{JS} extended Bernstein's result to $\mathbb{R}^{n}$ for $2\leq
n\leq7$ and a celebrated work of Giorgi, Giusti and Bombieri \cite{BGG} proved
that Bernstein's theorem is false in $\mathbb{R}^{n}$ for $n\geq8.$ A full
characterization of entire minimal graphs of $\mathbb{R}^{n}$ was given in
terms of the gradient at infinity of the solutions: Indeed, J. Moser proved
that an entire solution for the MSE in $\mathbb{R}^{n},$ $n\geq2,$ is an
affine function if and only if the norm of the gradient of the solution is
bounded (corollary of Theorem 6 of \cite{Mo}).

Moser's result is not true in the hyperbolic space: There are entire solutions
of the MSE in $\mathbb{H}^{n},$ $n\geq2,$ which assume a prescribed continuous
non-constant value at infinity and with bounded gradient (it is a consequence,
for instance, of Theorem 3.14 of \cite{RT})$.$ In this note we prove that if
the norm of the gradient converges exponentially to zero, then the solution
must be constant.

Our main result applies to the solutions of a quite broad class of partial
differential equations, according to the following definitions, and is optimal
in this class. Although we don't have a proof, we believe that the exponential
decay is also optimal for the MSE.

\begin{definition}
\label{def1}We say that a partial differential equation (PDE) on a complete
Riemannian manifold $M$ satisfies the \emph{comparison principle} if, given a
bounded domain $\Omega\subset M,$ if $u$ and $v$ are solutions of the PDE in
$\Omega\ $and $u\leq v$ in $\partial\Omega$ that is%
\[
\limsup_{k}\left(  u(x_{k})-v(x_{k})\right)  \leq0
\]
for any sequence $x_{k}\in\Omega$ that leaves any compact subset of $\Omega,$
then $u\leq v$ in $\Omega.$
\end{definition}

\begin{definition}
\label{def2}We say that a PDE in a Riemannian manifold $M$ is a
\emph{geometric }PDE\emph{ }if,$\ $whenever $u$ is a solution of the PDE on a
domain $\Omega$ of $M$, the composition $u_{\phi}:=u\circ\phi$ is also
solution on $\phi^{-1}\left(  \Omega\right)  $ for any isometry $\phi$ of $M.$
\end{definition}

Any PDE depending only on $u,$ $\nabla u,$ $\nabla^{2}u$ and other geometric
operators as $\Delta$ and $\operatorname{div}$ are invariant by isometries and
hence are geometric PDE's. Those of the form%
\[
\operatorname{div}\left(  \frac{a\left(  \left\Vert \nabla u\right\Vert
\right)  }{\left\Vert \nabla u\right\Vert }\nabla u\right)  +C=0,\text{
}C\text{ constant},
\]
where $a\in C^{1}\left(  \left[  0,\infty\right)  \right)  $, $a^{\prime}>0,$
$a\left(  0\right)  =0$ also satisfy the comparison principle (Proposition 3.1
of \cite{RT}). This class includes the minimal surface equation with%
\[
a\left(  s\right)  =\frac{s}{\sqrt{1+s^{2}}}%
\]
and the $p-$Laplace PDE where
\[
a\left(  s\right)  =s^{p-1},\text{ }p>1.
\]

Also the solutions of PDE's of the form $F\left(  \nabla^{2}u\right)  =0$
studied in \cite{CNS} satisfy the comparison principle (Lemma B of
\cite{CNS}). Comparison principles for more general geometric PDE's of the
form $F(u,\nabla u,\nabla^{2}u)=0$ are studied in \cite{PS}. On another hand,
a PDE depending on given function on $M$ may not be geometric. For example
$\Delta u=f$ is geometric if and only if $f$ is invariant by the isometries of
$M.$ Hence, if $M$ is a homogeneous manifold then $\Delta u=f$ is geometric if
and only if $f$ is constant. In general, PDE's\ depending only on the
differentiable structure of $M$ will not be geometric as $X\left(  X\left(
u\right)  \right)  =0,$ where $X$ is a vector field on $M.$ We prove:

\begin{theorem}
\label{hyp} Let $\mathbb{H}^{n}$ be the hyperbolic space of constant sectional
curvature $-1,$ $n\geq2.$ Let $u\in C^{1}\left(  \mathbb{H}^{n}\right)  $ be a
solution of a geometric PDE satisfying the comparison principle. If
\[
\limsup_{R\rightarrow\infty}\left(  e^{R}\sup_{S_{R}}\left\Vert \nabla
u\right\Vert \right)  =0,
\]
where $S_{R}$ is a geodesic sphere of $\mathbb{H}^{n}$ centered at fixed point
$o\in\mathbb{H}^{n}$ with radius $R,$ then $u$ is constant. Moreover, given
$C>0,$ there is a bounded non constant harmonic function $v\in C^{\infty
}\left(  \mathbb{H}^{n}\right)  $ such that
\[
\lim_{R\rightarrow\infty}\left(  e^{R}\sup_{S_{R}}\left\Vert \nabla
v\right\Vert \right)  =C.
\]

\end{theorem}

The first part of Theorem \ref{hyp} is actually a consequence of a more
general result that applies to a Lie group with a left invariant metric and to
the solutions of a broader class of partial differential equations which we
call \emph{left invariant }PDE's. To state it we first recall and introduce
some notations and definitions.

Let $G$ be a Lie group, $e$ its neutral element and $\mathfrak{g}=T_{e}G$ its
Lie algebra. Given $g\in G$ denote by $R_{g}$ and $L_{g}$ the right and left
translations of $G,$ $R_{g}(x)=xg$ and $L_{g}(x)=gx,$ $x\in G,$ and by
$C_{g}=L_{g}\circ R_{g}^{-1}$ the conjugation. Let $\operatorname*{Ad}%
\nolimits_{g}=d\left(  C_{g}\right)  _{e}:\mathfrak{g}\rightarrow\mathfrak{g}$
be the adjoint map of $G.$ Given an inner product $\left\langle \text{ ,
}\right\rangle _{e}$ in $\mathfrak{g}$, let $\left\langle \text{ ,
}\right\rangle $ be the left invariant metric of $G$ determined by
$\left\langle \text{ , }\right\rangle _{e},$ namely%
\[
\left\langle u,v\right\rangle _{g}:=\left\langle d(L_{g}^{-1})_{g}%
u,d(L_{g}^{-1})_{g}v\right\rangle _{e},\text{ \ }g\in G,\text{ \ }u,v\in
T_{g}G.
\]

\begin{definition}
We say that a PDE in $G$ is a \emph{left invariant }PDE\emph{ }if,$\ $whenever
$u$ is a solution of the PDE on a domain $\Omega$ of $G$, the composition
$u_{g}:=u\circ L_{g}$ is also solution on $L_{g}^{-1}\left(  \Omega\right)  $
for any $g\in G.$
\end{definition}

Clearly geometric PDE's on $G$ are left invariant. However, the converse is
not true. For example, considering $\mathbb{R}^{n}$ as a commutative Lie
group, the PDE $\Delta u=X(u),$ where $X$ is a vector field of $\mathbb{R}%
^{n}$ is left invariant but not geometric in general. We prove:

\begin{theorem}
\label{mpg} Let $\Omega$ be a bounded $C^{1}$ domain on a Lie group $G$ with a
left invariant metric and let $u\in C^{1}\left(  \overline{\Omega}\right)  $
be a solution of a left invariant PDE satisfying the comparison principle on
$\Omega$. Then
\begin{equation}
\sup_{\Omega}\left\Vert \nabla u\right\Vert \leq\sup_{g\in\Omega\cup
\Omega^{-1}}\left\Vert \operatorname*{Ad}\nolimits_{g}\right\Vert
\sup_{\partial\Omega}\left\Vert \nabla u\right\Vert , \label{max}%
\end{equation}
where $\Omega^{-1}=\{g^{-1}\ |\ g\in\Omega\}$.
\end{theorem}

Clearly $\left\Vert \operatorname*{Ad}\nolimits_{g}\right\Vert =1$ for any
$g\in G$ if the metric of $G$ is bi-invariant. In this case, then, the
solution satisfies the maximum principle for the gradient that is, it holds
the equality in (\ref{max}).

\begin{corollary}
\label{mt} Let $u\in C^{1}\left(  G\right)  $ be a solution of a left
invariant PDE satisfying the comparison principle on a non compact Lie group
$G$ with a left invariant metric. If
\[
\limsup_{R\rightarrow\infty}\left(  \sup_{g\in B_{R}}\left\Vert
\operatorname*{Ad}\nolimits_{g}\right\Vert \sup_{S_{R}}\left\Vert \nabla
u\right\Vert \right)  =0,
\]
where $B_{R}$ is the geodesic ball of $G$ centered at $e$ with radius $R$ and
$S_{R}=\partial B_{R},$ then $u$ is constant.
\end{corollary}

We prove that if $G$ has negative sectional curvature then the norm of the
adjoint map satisfies the strong maximum principle (Proposition \ref{ad}).

A well known paper of J. Milnor \cite{Mi} studies Lie groups with a left
invariant metric. More recently, $3-$dimensional Lie groups with a left
invariant metric have been classified and an explicit description is given in
\cite{MP}. Finally we mention that it follows by the Iwasava decomposition
that any symmetric space of noncompact type is isometric to a Lie group with a
left invariant metric \cite{H}.

\section{Preliminary facts, proofs of Theorem \ref{mpg} and Corollary
\ref{mt}}

\qquad\ We introduce some notation and prove some preliminary results to be
used in Theorem \ref{mpg} and Corollary \ref{mt}. Given $h$ and $g$ in $G$, we
denote
\[
R(g)_{h}=\Vert(R_{g})_{\star,h}\Vert
\]
where $(R_{g})_{\star,h}$ is the derivative of $R_{g}$ at the point $h.$

\begin{lemma}
The quantity $R(g)_{h}$ does not depend on $h$.
\end{lemma}

So, from now on we denote it
\[
R(g)=\Vert(R_{g})_{\star,e}\Vert
\]
and introduce, for a subset $S$ of $G$
\[
R_{S}=\max_{g\in S}R(g).
\]

\begin{proof}
We first notice that the left and right translations commute, i.e., if
$g_{1},g_{2},h\in G$,%

\[
L_{g_{1}}\circ R_{g_{2}}(h)=R_{g_{2}}\circ L_{g_{1}}(h)=g_{1}hg_{2}.
\]

Let $h,g\in G$. Then
\[
R(g)_{h}^{2}=\sup_{X\in T_{h}G,\Vert X\Vert=1}\langle(R_{g})_{\star}%
X,(R_{g})_{\star}X\rangle_{hg}%
\]%
\[
=\sup_{X\in T_{h}G,\Vert X\Vert=1}\langle(L_{(hg)^{-1}})_{\star}(R_{g}%
)_{\star}X,(L_{(hg)^{-1}})_{\star}(R_{g})_{\star}X\rangle_{e}%
\]%
\[
=\sup_{X\in T_{h}G,\Vert X\Vert=1}\langle(L_{g^{-1}})_{\star}(L_{h^{-1}%
})_{\star}(R_{g})_{\star}X,(L_{g^{-1}})_{\star}(L_{h^{-1}})_{\star}%
(R_{g})_{\star}X\rangle_{e}%
\]%
\begin{equation}
=\sup_{X\in T_{h}G,\Vert X\Vert=1}\langle(L_{g^{-1}})_{\star}(R_{g})_{\star
}(L_{h^{-1}})_{\star}X,(L_{g^{-1}})_{\star}(R_{g})_{\star}(L_{h^{-1}})_{\star
}X\rangle_{e}. \label{avant-derniere}%
\end{equation}

For $X\in T_{h}G$, $Z=(L_{h^{-1}})_{\star}X$ belongs to $T_{e}G$; moreover
\[
\Vert X\Vert_{h}=\Vert(L_{h^{-1}})_{\star}X\Vert_{1}%
\]
and we can rewrite (\ref{avant-derniere}) as
\[
(\ref{avant-derniere})=\sup_{Z\in T_{e}G,\Vert Z\Vert=1}\langle(L_{g^{-1}%
})_{\star}(R_{g})_{\star}Z,(L_{g^{-1}})_{\star}(R_{g})_{\star}Z\rangle_{e}%
\]%
\[
=\sup_{Z\in T_{e}G,\Vert Z\Vert=1}\langle(R_{g})_{\star}Z,(R_{g})_{\star
}Z\rangle_{g}=(R_{g})_{\star,e}.
\]

\end{proof}

\begin{lemma}
\label{RAd}If $g\in G$,
\[
R(g)=\Vert\operatorname*{Ad}\nolimits_{g}\Vert.
\]

\end{lemma}

\begin{proof}
If $X\in T_{e}G$, by definition of the metric
\[
\langle(R_{g})_{\star}X,(R_{g})_{\star}X\rangle_{g}=\langle(L_{g^{-1}}%
)_{\star}(R_{g})_{\star}X,(L_{g^{-1}})_{\star}(R_{g})_{\star}X\rangle_{e}%
\]%
\[
=\langle\operatorname*{Ad}\nolimits_{g}(X),\operatorname*{Ad}\nolimits_{g}%
(X)\rangle_{e}.
\]

\end{proof}

Denote by $d$ the Riemannian distance in $G.$ The next lemma is elementary and
its proof is therefore omitted:

\begin{lemma}
\label{g1}Let $\Omega$ be a $C^{1}$ open subset of $G.$ Let $u\in C^{1}\left(
\overline{\Omega}\right)  $ and $x_{0}\in\overline{\Omega}$ be given and
assume that $\left\Vert \nabla u\right\Vert \left(  x_{0}\right)  \neq0.$ Let
$\gamma:\left[  0,\varepsilon\right)  \rightarrow\Omega$ be an arc length
geodesic such that $\gamma(0)=x_{0}$ and%

\[
\gamma^{\prime}(0)=\frac{\nabla u\left(  x_{0}\right)  }{\left\Vert \nabla
u\left(  x_{0}\right)  \right\Vert }\text{ or }\gamma^{\prime}(0)=-\frac
{\nabla u\left(  x_{0}\right)  }{\left\Vert \nabla u\left(  x_{0}\right)
\right\Vert }.
\]
Then
\[
\left\Vert \nabla u\right\Vert \left(  x_{0}\right)  =\lim_{t\rightarrow
0}\frac{\left\vert u\left(  \gamma(t)\right)  -u\left(  x_{0}\right)
\right\vert }{d(\gamma(t),x_{0})}.
\]

\end{lemma}

\begin{lemma}
\label{inegalite} Let $a,b,g$ be elements in $G$. Denoting by $d$ the distance
on $G$, we have
\[
d(a.g,b.g)\leq R(g) d(a,b)
\]
\[
d(a,b)\leq R(g^{-1}) d(a.g,b.g).
\]

\end{lemma}

\begin{proof}
If $\gamma$ is a minimizing geodesic between $a$ and $b\ $then
$d(a,b)=\operatorname*{length}(\gamma).$ Also, since $R_{g}\circ\gamma$ is a
path between $a.g$ and $b.g$,
\[
d(a.g,b.g)\leq\operatorname*{length}(R_{g}\circ\gamma)\leq
R(g)\operatorname*{length}(\gamma)=R(g)d(a,b).
\]
We have thus proved the first inequality of the lemma. We derive
\[
d(a,b)=d\left(  (a.g)g^{-1},(b.g)g^{-1}\right)  \leq R(g^{-1})d(a.g,b.g)
\]
which proves the second inequality.
\end{proof}

\begin{proof}
[Proof of Theorem \ref{mpg}]The proof is an extension of Lemma 12.7 of
\cite{Gi} (see also \cite{W}). Let $\Omega$ be a bounded $C^{1}$ domain on $G$
with and let $u\in C^{1}\left(  \overline{\Omega}\right)  $ be a solution of a
left invariant PDE satisfying the comparison principle. Take
\[
k=\sup\left\{  \frac{\left\vert u(x)-u(z)\right\vert }{d(x,z)},\text{ }%
x\in\Omega,\text{ }z\in\partial\Omega\right\}  .
\]
From Lemma \ref{g1} it is then enough to prove that%

\begin{equation}
\left\vert u(x)-u(y)\right\vert \leq kd(x,y) \label{en}%
\end{equation}
for all $x,y\in\Omega.$ Given $x_{1},x_{2}\in\Omega,$ set $z=x_{1}\cdot
x_{2}^{-1},$
\[
\Omega_{z}=\left\{  z^{-1}\cdot x\in G\text{
$\vert$
}x\in\Omega\right\}  =\left\{  x\in G\text{
$\vert$
}z\cdot x\in\Omega\right\}
\]
and define $u_{z}\in C^{1}\left(  \overline{\Omega}_{z}\right)  $ by
$u_{z}(x)=u(zx).$ We have $\Omega\cap\Omega_{z}\neq\varnothing$ since $x_{2}$
belongs to both. By the comparison principle%
\[
\sup\left\{  u(x)-u_{z}(x)\text{
$\vert$
}x\in\Omega\cap\Omega_{z}\right\}  =\sup\left\{  u(x)-u_{z}(x)\text{
$\vert$
}x\in\partial\left(  \Omega\cap\Omega_{z}\right)  \right\}  .
\]
Then, in particular%
\[
u(x_{2})-u(x_{1})=u(x_{2})-u_{z}(x_{2})\leq\sup\left\{  u(x)-u(zx)\text{
$\vert$
}x\in\partial\left(  \Omega\cap\Omega_{z}\right)  \right\}  .
\]
Let $x_{0}\in\partial\left(  \Omega\cap\Omega_{z}\right)  $ be such that
\[
u(x_{0})-u(zx_{0})=\sup\left\{  u(x)-u(zx)\text{
$\vert$
}x\in\partial\left(  \Omega\cap\Omega_{z}\right)  \right\}  .
\]
If $x_{0}\in\partial\left(  \Omega\cap\Omega_{z}\right)  $ then either
$x_{0}\in\partial\Omega$ or $zx_{0}\in\partial\Omega.$ It then follows from
the hypothesis that%
\[
u(x_{2})-u(x_{1})\leq kd(x_{0},zx_{0}).
\]
Now, using Lemma \ref{inegalite} twice, we have
\[
d(x_{0},(x_{1}x_{2}^{-1})x_{0})\leq R(x_{0})d(1,x_{1}x_{2}^{-1})\leq
R(x_{0})R(x_{2}^{-1})d(x_{2},(x_{1}x_{2}^{-1})x_{2})
\]%
\[
=R(x_{0})R(x_{2}^{-1})d(x_{2},x_{1}).
\]

\end{proof}

\begin{proof}
[\textbf{Proof of Corollary \ref{mt}}]Given $R>0$, if $g\in S_{R}$ then
\[
R=d(e,g)=d(g^{-1},e)=d(e,g^{-1})
\]
so that $S_{R}^{-1}=S_{R}.$ It follows that $B_{R}^{-1}=B_{R},$ where $B_{R}$
is the geodesic ball centered at $e$ with radius $R.$ Then, if $u\in C^{1}(G)$
is an entire solution satisfying the hypothesis of Corollary \ref{mt}, it
follows from Theorem \ref{mpg} that%
\[
\max_{B_{R}}\left\Vert \nabla u\right\Vert \leq\max_{g\in B_{R}}\left\Vert
\operatorname*{Ad}\nolimits_{g}\right\Vert \max_{S_{R}}\left\Vert \nabla
u\right\Vert ,
\]
from which the proof follows.
\end{proof}

We close this section by proving a strong maximum principle for the adjoint map:

\begin{proposition}
\label{ad}Let $G$ be a Lie group with negative sectional curvature and let
$\Lambda$ be any open subset of $G.$ Then%
\[
\left\Vert \operatorname*{Ad}\nolimits_{h}\right\Vert <\sup_{g\in
\partial\Lambda}\left\Vert \operatorname*{Ad}\nolimits_{g}\right\Vert
\]
for all $h\in\Lambda.$ In particular
\[
\sup_{g\in\Lambda}\left\Vert \operatorname*{Ad}\nolimits_{g}\right\Vert
=\sup_{g\in\partial\Lambda}\left\Vert \operatorname*{Ad}\nolimits_{g}%
\right\Vert .
\]

\end{proposition}

\begin{proof}
By contradiction, assume that
\[
\left\Vert \operatorname*{Ad}\nolimits_{h}\right\Vert =\sup_{g\in\Lambda
}\left\Vert \operatorname*{Ad}\nolimits_{g}\right\Vert
\]
for some $h\in\Lambda.$ There exists $x\in T_{e}G,$ $\left\Vert x\right\Vert
=1,$ such that $\left\Vert \operatorname*{Ad}\nolimits_{h}\right\Vert
=\left\Vert \operatorname*{Ad}\nolimits_{h}\left(  x\right)  \right\Vert .$
Let $X$ be the right invariant vector field of $G$ such that $X\left(
e\right)  =x.$ Choose $u\in T_{h}G,$ $\left\Vert u\right\Vert =1,$ such that
$u\neq X\left(  h\right)  $ and let $\gamma\left(  t\right)  ,$ $t\geq0,\ $be
the geodesic parametrized by arc length in $G$ such that $\gamma\left(
0\right)  =h$ and $\gamma^{\prime}\left(  0\right)  =u.$ Then
\begin{align}
\left.  \frac{d}{dt}\left\Vert \operatorname*{Ad}\nolimits_{\gamma\left(
t\right)  }\left(  x\right)  \right\Vert ^{2}\right\vert _{t=0} &
=0\label{d}\\
\left.  \frac{d^{2}}{dt^{2}}\left\Vert \operatorname*{Ad}\nolimits_{\gamma
\left(  t\right)  }\left(  x\right)  \right\Vert ^{2}\right\vert _{t=0} &
\leq0.\label{d2}%
\end{align}

We have%
\begin{align*}
\left\Vert \operatorname*{Ad}\nolimits_{\gamma\left(  t\right)  }\left(
x\right)  \right\Vert ^{2}  &  =\left\langle \operatorname*{Ad}%
\nolimits_{\gamma\left(  t\right)  }\left(  x\right)  ,\operatorname*{Ad}%
\nolimits_{\gamma\left(  t\right)  }\left(  x\right)  \right\rangle \\
&  =\left\langle d\left(  L_{\gamma\left(  t\right)  }^{-1} \right)
_{\gamma\left(  t\right)  }\left(  d\left(  R_{\gamma\left(  t\right)
}\right)  _{e}\right)  \left(  x\right)  ,d\left(  L_{\gamma\left(  t\right)
}^{-1}\right)  _{\gamma\left(  t\right)  } \left(  d\left(  R_{\gamma\left(
t\right)  }\right)  _{e}\right)  \left(  x\right)  \right\rangle \\
&  =\left\langle d\left(  R_{\gamma\left(  t\right)  }\right)  _{e}\left(
x\right)  ,d\left(  R_{\gamma\left(  t\right)  }\right)  _{e}\left(  x\right)
\right\rangle ,\text{ }t\geq0.
\end{align*}

Then%
\[
\left\Vert \operatorname*{Ad}\nolimits_{\gamma\left(  t\right)  }\left(
x\right)  \right\Vert ^{2}=\left\langle X\left(  \gamma\left(  t\right)
\right)  ,X\left(  \gamma\left(  t\right)  \right)  \right\rangle =\left\Vert
X\left(  \gamma\left(  t\right)  \right)  \right\Vert ^{2},\text{ }t\geq0.
\]

Since right invariant vector fields are Killing fields and Killing fields
restricted to geodesics are Jacobi vector fields,
\[
J\left(  t\right)  :=X\left(  \gamma\left(  t\right)  \right)
\]
is a Jacobi field along $\gamma,$ $t\geq0.$ Thus, $\left\Vert
\operatorname*{Ad}\nolimits_{\gamma\left(  t\right)  }\left(  x\right)
\right\Vert ^{2}$ is equal to the square of the norm of the Jacobi field
$J\left(  t\right)  $ along $\gamma$ satisfying the initial conditions%
\begin{align*}
J\left(  0\right)   &  =X\left(  h\right) \\
J^{\prime}\left(  0\right)   &  =\nabla_{\gamma^{\prime}\left(  0\right)
}X=\nabla_{u}X.
\end{align*}

We have, from (\ref{d})
\[
\left.  \frac{d}{dt}\left\Vert J\left(  t\right)  \right\Vert ^{2}\right\vert
_{t=0}=2\left\langle J^{\prime}\left(  0\right)  ,J\left(  0\right)
\right\rangle =0.
\]
And, using the Jacobi equation,%
\begin{align*}
\frac{d^{2}}{dt^{2}}\left\Vert J\left(  t\right)  \right\Vert ^{2}  &
=2\left\langle J^{\prime\prime}\left(  t\right)  ,J\left(  t\right)
\right\rangle +2\left\langle J^{\prime}\left(  t\right)  ,J^{\prime}\left(
t\right)  \right\rangle \\
&  =-2\left\langle R(\gamma^{\prime},J)\gamma^{\prime},J\left(  t\right)
\right\rangle +2\left\Vert J^{\prime}\left(  t\right)  \right\Vert ^{2}.
\end{align*}

Since $u\neq X\left(  h\right)  $ the vector fields $J(t)=X(\gamma(t))$ and
$\gamma^{\prime}(t)$ are linearly independent along $\gamma$ and we then have%
\[
\frac{d^{2}}{dt^{2}}\left\Vert J\left(  t\right)  \right\Vert ^{2}=-2K\left(
\gamma^{\prime},J\right)  \left\Vert \gamma^{\prime}\wedge J\right\Vert ^{2}
+2\left\Vert J^{\prime}\right\Vert ^{2}.
\]

Since $K<0$ it follows that%
\[
\left.  \frac{d^{2}}{dt^{2}}\left\Vert J\left(  t\right)  \right\Vert
^{2}\right\vert _{t=0}>0
\]
contradicting (\ref{d2}). This proves the proposition.
\end{proof}

\section{The hyperbolic space. Proof of Theorem \ref{hyp}.}

We begin by calculating the norm of the adjoint map in the hyperbolic space to
apply Corollary \ref{mt}. We present an explicit construction of the Lie group
structure of the hyperbolic space, that comes from the Iwasawa decomposition
\cite{H}.

In the half-space model
\[
\mathbb{H}^{n}=\left\{  \left(  x_{1},...,x_{n}\right)  ,\text{
$\vert$
}x_{n}>0\right\}  ,\text{ }ds^{2}=\frac{\delta_{ij}}{x_{n}^{2}},
\]
of the hyperbolic $n-$dimensional space, $n\geq2,$ given $s>0$ and $t:=\left(
t_{1},...,t_{n-1}\right)  \in\mathbb{R}^{n-1}$, define
\begin{align*}
a_{s},n_{t}  &  :\mathbb{H}^{n}\rightarrow\mathbb{H}^{n},\\
a_{s}\left(  x_{1},...,x_{n}\right)   &  =s\left(  x_{1},...,x_{n}\right) \\
n_{t}\left(  x_{1},...,x_{n}\right)   &  =\left(  x_{1}+t_{1},...,x_{n-1}%
+t_{n-1},x_{n}\right)  .
\end{align*}
Set
\[
G:=AN=\left\{  a_{s}\circ n_{t}\text{
$\vert$
}s>0,\text{ }t\in\mathbb{R}^{n-1}\right\}
\]
where%
\begin{align*}
A  &  =\left\{  a_{s},\text{ }s>0\right\} \\
N  &  =\left\{  n_{t},\text{ }t\in\mathbb{R}^{n-1}\right\}  .
\end{align*}

Given $p\in\mathbb{H}^{n}$ there is one and only one $g_{p}\in G\subset
\operatorname*{Iso}\left(  \mathbb{H}^{n}\right)  $ such that $p=g_{p}\left(
\left(  0,...,0,1\right)  \right)  .$ Indeed:\ If $p=\left(  x_{1}%
,...,x_{n}\right)  \in\mathbb{H}^{n}$ define $n:=n_{\left(  x_{1}%
,...,x_{n-1},0\right)  },a:=a_{x_{n}}\in\operatorname*{Iso}\left(
\mathbb{H}^{n}\right)  ,$ that is%
\begin{align*}
n\left(  z_{1},...,z_{n}\right)   &  =\left(  z_{1}+x_{1},...,z_{n-1}%
+x_{n-1},z_{n}\right) \\
a  &  =x_{n}\left(  z_{1},...,z_{n}\right)  ,\text{ }\left(  z_{1}%
,...,z_{n}\right)  \in\mathbb{H}^{n}.
\end{align*}

Then, taking%
\begin{equation}
g_{p}=n\circ a \label{for}%
\end{equation}
we have%
\[
g_{p}\left(  0,...,0,1\right)  =n\left(  a\left(  0,...,0,1\right)  \right)
=n\left(  0,...,0,x_{n}\right)  =\left(  x_{1},...,x_{n}\right)  =p.
\]
Note that $n$ and $a$ are not uniquely determined by $p,$ but $g_{p}$ is.

One may see that with the operation%
\[
p\cdot q:=\left(  g_{p}\circ g_{q}\right)  \left(  \left(  0,...,0,1\right)
\right)
\]
$\mathbb{H}^{n}$ is a Lie group (a solvable Lie group indeed. This is a
general fact that holds for symmetric spaces of non compact type (\cite{H},
Chapter VI))$.$ Moreover, given $p\in\mathbb{H}^{n}$ we have, for any
$q\in\mathbb{H}^{n},$
\[
L_{p}\left(  q\right)  =p\cdot q=\left(  g_{p}\circ g_{q}\right)  \left(
\left(  0,...,0,1\right)  \right)  =g_{p}\left(  g_{q}\left(  \left(
0,...,0,1\right)  \right)  \right)  =g_{p}\left(  q\right)
\]
that is $L_{p}=g_{p}.$ Since $g_{q}\in\operatorname*{Iso}\left(
\mathbb{H}^{n}\right)  $ it follows that the left translation $L_{p}$ is an
isometry of $\mathbb{H}^{n}$ with respect to the hyperbolic metric that is,
the hyperbolic metric is left invariant with respect to the Lie group
structure of $\mathbb{H}^{n}.$

\begin{proposition}
\label{norad}Consider $\mathbb{H}^{n}$ as a Lie group and let $e$ be its
neutral element. If $B_{R}$ is the closed geodesic ball of $\mathbb{H}^{n}$
centered at $e$ with radius $R$ then
\begin{equation}
\max_{g\in B_{R}}\left\Vert \operatorname*{Ad}\nolimits_{g}\right\Vert =\cosh
R+\sinh R. \label{nad}%
\end{equation}

\end{proposition}

\begin{proof}
We claim that it is enough to consider the $2-$dimensional case. Indeed: Let
$g\in B_{R}$ and $x\in T_{e}\mathbb{H}^{n}$ be such that%
\[
\left\Vert \operatorname*{Ad}\nolimits_{g}\left(  x\right)  \right\Vert
=\max_{h\in B_{R}}\left\Vert \operatorname*{Ad}\nolimits_{h}\right\Vert .
\]
Let $y\in T_{e}\mathbb{H}^{n}$ be a nonzero vector tangent to the geodesic
from $e$ to $g.$ There exists a totally geodesic hyperbolic plane
$\mathbb{H}^{2}$ of $\mathbb{H}^{n}$ such that $e\in\mathbb{H}^{2}$ and
$x,y\in T_{e}\mathbb{H}^{2}.$ Since $\mathbb{H}^{2}$ is a Lie subgroup of
$\mathbb{H}^{n}$ we have
\[
\left\Vert \operatorname*{Ad}\nolimits_{g}\left(  x\right)  \right\Vert
=\max_{h\in D_{R}}\left\Vert \operatorname*{Ad}\nolimits_{h}\right\Vert
\]
where $D_{R}=B_{R}\cap\mathbb{H}^{2}$ is the closed geodesic disk centered at
$e$ with radius $R.$ It is clear that the maximum of the norm of the adjoint
map does not depend on hyperbolic plane containing $e.$ This proves our claim.

Considering the half-plane model for $\mathbb{H}^{2}$ with $e=\left(
0,1\right)  ,$ given $p\in\mathbb{H}^{2},$ the conjugation $C_{p}%
:\mathbb{H}^{2}\rightarrow\mathbb{H}^{2}$ is given by%
\begin{equation}
C_{_{p}}\left(  q\right)  =\left(  g_{p}\circ g_{q}\circ g_{p}^{-1}\right)
\left(  0,1\right)  =g_{p}\left(  g_{q}\left(  g_{p}^{-1}\left(  0,1\right)
\right)  \right)  . \label{conj}%
\end{equation}
Expanding (\ref{conj}) we arrive to%
\[
C_{\left(  x,y\right)  }\left(  z,w\right)  =\left(  -xw+yz+x,w\right)
\]
from which we obtain, at a given $X:=\left(  a,b\right)  \in T_{\left(
0,1\right)  }\mathbb{H}^{2},$%
\[
\operatorname*{Ad}\nolimits_{\left(  x,y\right)  }\left(  a,b\right)
=d\left(  C_{\left(  x,y\right)  }\right)  _{e}\left(  a,b\right)  =\left(
-xb+ya,b\right)  .
\]
Since at $\left(  0,1\right)  $ the hyperbolic metric coincides with the
Euclidean metric and since $\operatorname*{Ad}\nolimits_{\left(  x,y\right)
}\left(  X\right)  \in T_{\left(  0,1\right)  }\mathbb{H}^{2}$ we obtain%
\begin{equation}
\left\Vert \operatorname*{Ad}\nolimits_{\left(  x,y\right)  }\left(  X\right)
\right\Vert =\sqrt{\left(  -xb+ya\right)  ^{2}+b^{2}}. \label{a}%
\end{equation}
Since the identity of $\mathbb{R}_{+}^{2}$ is a conformal map between the
Euclidean and hyperbolic geometries, the hyperbolic geodesic circle in the
half plane model of $\mathbb{H}^{2}$ centered at $(0,1)$ with hyperbolic
radius $R$ is also an Euclidean circle. Moreover, because an Euclidean
symmetry with respect to a vertical line is also a hyperbolic isometry, the
Euclidean center of the circle is at some point $\left(  0,y_{0}\right)  $ of
the vertical line $x=0.$ The Euclidean circle is parametrized by
\begin{equation}
\left(  0,y_{0}\right)  +r\left(  \cos\theta,\sin\theta\right)  ,\text{
}\theta\in\left[  0,2\pi\right)  , \label{1}%
\end{equation}
where $r$ is the Euclidean radius. Since $\gamma\left(  t\right)  =\left(
0,e^{t}\right)  $ is an arc length hyperbolic geodesic such that
$\gamma\left(  0\right)  =\left(  0,1\right)  $ we have $t=d_{\mathbb{H}^{2}%
}\left(  \left(  0,1\right)  ,\left(  0,e^{\pm t}\right)  \right)  .$ In
particular
\[
R=d_{\mathbb{H}^{2}}\left(  \left(  0,1\right)  ,\left(  0,e^{R}\right)
\right)  =d_{\mathbb{H}^{2}}\left(  \left(  0,1\right)  ,\left(
0,e^{-R}\right)  \right)  .
\]
It follows that the points $\left(  0,e^{R}\right)  $ and $\left(
0,e^{-R}\right)  $ are both in the hyperbolic circle centered at $\left(
0,1\right)  $ and with hyperbolic radius $R.$ Then the Euclidean circle must
contain both points $\left(  0,e^{R}\right)  $ and $\left(  0,e^{-R}\right)
.$ Since these points are in the same vertical straight line, the Euclidean
center of Euclidean circle is
\[
\frac{\left(  0,e^{R}\right)  +\left(  0,e^{-R}\right)  }{2}=\left(  0,\cosh
R\right)  .
\]
And the Euclidean radius $r$ of this hyperbolic circle is just half of the
Euclidean distance between the points $\left(  0,e^{R}\right)  $ and $\left(
0,e^{-R}\right)  $ that is,%
\[
r=\frac{1}{2}d_{\mathbb{R}^{2}}\left(  \left(  0,e^{R}\right)  ,\left(
0,e^{-R}\right)  \right)  =\frac{e^{R}-e^{-R}}{2}=\sinh R.
\]
It follows from (\ref{1}) that the geodesic disk $D_{R}$ of $\mathbb{H}^{2}$
centered at $\left(  0,1\right)  $ with radius $R$ is given by%
\[
D_{R}=\left\{  \left(  \sinh r\cos\theta,\cosh r+\sinh r\sin\theta\right)
\text{
$\vert$
}\theta\in\mathbb{R},\text{ }0\leq r\leq R\right\}  .
\]

From (\ref{a}),%
\[
\left\Vert \operatorname*{Ad}\nolimits_{\left(  x,y\right)  }\left(  X\right)
\right\Vert ^{2}=\left(  \cosh r+\sinh r\sin\theta\right)  ^{2}\left(
\frac{-b\sinh r\cos\theta}{\cosh r+\sinh r\sin\theta}+a\right)  ^{2}+b^{2}%
\]
and we may see\ that the biggest value of right hand side of the last
inequality occurs at $\theta=\pi/2$ and $a=1,$ $b=0.$ From Proposition
\ref{ad} (which, in the case of the hyperbolic space, can also be directly
confirmed from the expression above) we obtain (\ref{nad}).
\end{proof}

\begin{proof}
[Proof of Theorem \ref{hyp}]The first part of the proof is a direct
consequence of Corollary \ref{mt} and Proposition \ref{norad}. For the second
part consider a polar coordinate system $(r,\theta,\varphi_{1},\dots
,\varphi_{n-2})$ of $\mathbb{H}^{n}$ centered at a point $o\in\mathbb{H}^{n}$
that is, $\Theta=\left(  \theta,\varphi_{1},...,\varphi_{n-2}\right)  $ are
spherical coordinates of the unit sphere $\mathbb{S}^{n-1}$ centered at the
origin of $T_{o}\mathbb{H}^{n}$ and $p\in\mathbb{H}^{n}\backslash\left\{
o\right\}  $ is parametrized by%
\[
p=\exp_{o}\left(  r\Theta\right)  ,\text{ }r>0,\text{ }\Theta\in
\mathbb{S}^{n-1}.
\]
Let $v$ be the function that depends only on $r$ and $\theta$ given by
\[
v(r,\theta)=\frac{C(n-1)}{2}\cdot\frac{\displaystyle\int_{0}^{r}(\sinh
s)^{n-1}ds}{(\sinh r)^{n-1}}\cos\theta.
\]
Remind that in this coordinate system, the metric has the form
\begin{align*}
ds^{2}=dr^{2}+\sinh^{2}rd\theta &  +\sinh^{2}r\sin^{2}\theta d\varphi
_{1}+\sinh^{2}r\sin^{2}\theta\sin^{2}\varphi_{1}d\varphi_{2}\\[5pt]
&  +\dots+\sinh^{2}r\sin^{2}\theta\sin^{2}\varphi_{1}\dots\sin^{2}%
\varphi_{n-3}d\varphi_{n-2}%
\end{align*}
and, since $v=v(r,\theta)$, the Laplacian of $v$ can be expressed by
\[
\Delta v=\frac{\partial^{2}v}{\partial\,r^{2}}+(n-1)\coth r\frac{\partial
v}{\partial\,r}+(n-2)\frac{\cot\theta}{\sinh^{2}r}\cdot\frac{\partial
v}{\partial\,\theta}+\frac{1}{\sinh^{2}r}\cdot\frac{\partial^{2}v}%
{\partial\,\theta^{2}}.
\]
Hence, by a straight calculation, we have that $\Delta v=0$. Moreover, since
$v$ depends only on $r$ and $\theta$ and the fields $E_{1}:=\frac{\partial
}{\partial\,r}$ and $E_{2}:=\frac{\partial}{\partial\,\theta}$ are orthogonal,
we have
\begin{align*}
&  \Vert\nabla v(r,\theta)\Vert^{2}=\frac{|E_{1}(v)|^{2}}{\Vert E_{1}\Vert
^{2}}+\frac{|E_{2}(v)|^{2}}{\Vert E_{2}\Vert^{2}}=\left(  \frac{|\partial
_{r}v|^{2}}{1}+\frac{|\partial_{\theta}v|^{2}}{\sinh^{2}r}\right) \\[5pt]
&  =\left(  \frac{C(n-1)}{2}\right)  ^{2}\left(  1-(n-1)\displaystyle\cosh
r(\sinh r)^{-n}\int_{0}^{r}(\sinh s)^{n-1}ds\right)  ^{2}\cos^{2}\theta\\
&  +\left(  \frac{C(n-1)}{2}\right)  ^{2}\left(  \displaystyle(\sinh
r)^{-n}\int_{0}^{r}(\sinh s)^{n-1}ds\right)  ^{2}\sin^{2}\theta.
\end{align*}
Then,
\[
\Vert\nabla v(r,\theta)\Vert=\frac{C}{2(1+\cosh r)}\quad\text{ for any
}\;r\geq0\;\text{ if }\;n=2
\]
and
\[
\Vert\nabla v(r,\theta)\Vert\approx Ce^{-r}\sin\theta\quad\text{ as
}\;r\rightarrow+\infty\;\text{ if }\;n>2.
\]
In both cases
\[
\sup_{S_{R}}\Vert\nabla v(R,\theta)\Vert\approx Ce^{-R}\quad\text{as }\quad
R\rightarrow+\infty.
\]
Therefore, $v$ is a bounded non constant harmonic function such that
\[
\lim_{R\rightarrow\infty}\left(  e^{R}\sup_{S_{R}}\left\Vert \nabla
v\right\Vert \right)  =C,
\]
completing the proof.
\end{proof}

\newpage\noindent Ari Aiolfi \newline\noindent Universidade Federal de Santa
Maria\newline\noindent Brazil\newline\noindent ari.aiolfi@ufsm.br

\medskip

\noindent Leonardo Bonorino\newline\noindent Universidade Federal do Rio
Grande do Sul\newline\noindent Brazil\newline\noindent leonardo.bonorino@ufrgs.br

\medskip

\noindent Jaime Ripoll\newline\noindent Universidade Federal do Rio Grande do
Sul\newline\noindent Brazil\newline\noindent jaime.ripoll@ufrgs.br

\medskip

\noindent Marc Soret

\noindent Universit\'{e} de Tours \newline\noindent France

\noindent marc.soret@idpoisson.fr

\medskip

\noindent Marina Ville

\noindent Univ Paris Est Creteil, CNRS, LAMA, F-94010 Creteil \newline%
\noindent France

\noindent villemarina@yahoo.fr

\end{document}